\documentclass[11pt,reqno]{amsart}

\usepackage[T1]{fontenc}
\usepackage{lmodern}
\usepackage{amsmath,amssymb,mathtools}
\usepackage[hidelinks]{hyperref}

\hypersetup{
  pdftitle={Weighted Derivative Sums of a Gamma Quotient: Sun's Conjecture and Cyclotomic Specializations},
  pdfauthor={Shivam Nalin Patel},
  pdfsubject={Gamma quotients, weighted inverse-binomial sums, log-sine integrals, and cyclotomic polylogarithms},
  pdfkeywords={gamma quotient, derivative sum, inverse central binomial sum, log-sine integral, Dirichlet L-function, cyclotomic multiple polylogarithm}
}

\allowdisplaybreaks
\numberwithin{equation}{section}

\newtheorem{theorem}{Theorem}[section]
\newtheorem{proposition}[theorem]{Proposition}
\newtheorem{lemma}[theorem]{Lemma}
\newtheorem{corollary}[theorem]{Corollary}
\theoremstyle{remark}
\newtheorem{remark}[theorem]{Remark}

\newcommand{\dd}{\,\mathrm{d}}
\newcommand{\Cl}{\operatorname{Cl}}
\newcommand{\Ls}{\operatorname{Ls}}
\newcommand{\Gl}{\operatorname{Gl}}

\title[Weighted derivative sums of a gamma quotient]{Weighted derivative sums of a gamma quotient: Sun's conjecture and cyclotomic specializations}
\author{Shivam Nalin Patel}
\address{Independent Researcher, Newark, California, USA}
\email{patelshivam99@gmail.com}
\date{}


\begin{document}

\begin{abstract}
Let $f(x)=\Gamma(x)^2/(2\Gamma(2x))$ and set
$\lambda_\alpha=4\sin^2\alpha$ for $0<\alpha<\pi/2$.  We establish
an elementary parameter identity for a weighted translate of $f$ and
derive an explicit formula, valid at every derivative order, for the
associated weighted sums.  The coefficients satisfy an effective
recurrence in ordinary zeta values.  The unique unweighted
specialization $\alpha=\pi/6$ proves Conjecture~4.1 of Zhi-Wei Sun;
at the fourth order a depth-two value $\Gl_{4,1}(\pi/3)$ occurs.  The
construction complements general cyclotomic-multiple-zeta methods for
inverse-binomial harmonic sums by supplying a continuous master
identity, with concrete specializations at $\alpha=\pi/4$ and
$\alpha=\pi/3$.
\end{abstract}

\maketitle

\section{Introduction and statement of results}

Let $\chi_{-3}$ denote the primitive Dirichlet character modulo $3$,
so that
\[
 \chi_{-3}(n)=
 \begin{cases}
  0,& 3\mid n,\\
  1,& n\equiv 1\pmod 3,\\
 -1,& n\equiv 2\pmod 3,
 \end{cases}
\]
and write
\[
 L_{-3}(s)=L(s,\chi_{-3})
 =\sum_{n=1}^{\infty}\frac{\chi_{-3}(n)}{n^s}
 \qquad (\Re s>0).
\]
For $n\geq2$ and $m\geq1$, we use the notation
\begin{align}
 \Ls_n(\phi)
 &=-\int_0^\phi
   \log^{n-1}\left(2\sin\frac{x}{2}\right)\dd x,
   \qquad 0\leq\phi\leq2\pi, \label{eq:Ls-def}\\
 \Cl_{2m}(\phi)
 &=\sum_{n=1}^{\infty}\frac{\sin(n\phi)}{n^{2m}}. \label{eq:Cl-def}
\end{align}
For a composition $\boldsymbol{s}=(s_1,\ldots,s_d)$ with $s_1\geq2$, let
\begin{equation}\label{eq:MPL-def}
 \operatorname{Li}_{\boldsymbol{s}}(z)
 =\sum_{n_1>\cdots>n_d\geq1}
 \frac{z^{n_1}}{n_1^{s_1}\cdots n_d^{s_d}}.
\end{equation}
Following Borwein and Straub \cite{BorweinStraub2011}, we set
\begin{equation}\label{eq:Gl-def}
 \Gl_{4,1}(\theta)=\Im\operatorname{Li}_{4,1}(e^{i\theta}).
\end{equation}

In \cite[Conjecture~4.1, equations~(4.1)--(4.3)]{Sun2026}, Zhi-Wei
Sun considered
\begin{equation}\label{eq:f-def}
 f(x)=\frac{\Gamma(x)^2}{2\Gamma(2x)},\qquad x>0,
\end{equation}
and conjectured exact evaluations of its first three derivative sums
at the positive integers.  Since
\[
 f(k)=\frac{1}{k\binom{2k}{k}},
\]
the problem belongs naturally to the theory of inverse central-binomial
series.  Classical generating functions and special values in this area
go back at least to Lehmer and Zucker \cite{Lehmer1985,Zucker1985}.
More recently, Au developed systematic iterated-integral methods for
polylogarithmic integrals and binomial sums \cite{Au2020}, as well as a
cyclotomic-multiple-zeta framework with an accompanying computer-algebra
implementation \cite{Au2022}.  Zhou applied cyclotomic multiple zeta
values to several of Sun's series involving binomial coefficients and
harmonic numbers \cite{Zhou2023}; Sun and Zhou treated further
parameterized inverse-binomial families at many cyclotomic levels
\cite{SunZhou2024}; and Campbell, Glasser, and Zhou used Au's machinery
to evaluate additional inverse-binomial series \cite{CampbellGlasserZhou2024}.
These works make the occurrence of root-of-unity periods in the present
problem natural.

The contribution here is complementary.  Rather than expanding and
reducing each fixed-order harmonic sum separately, we retain a gamma
parameter, perform the summation before differentiation, and obtain one
elementary beta--hypergeometric identity.  It yields a continuous
weighted family, an explicit formula for every derivative order, and a
zeta-valued recurrence for the gamma-quotient coefficients.  To the best
of our knowledge, neither this master identity nor the resulting
all-orders derivative formula appears in the cited cyclotomic-MZV
literature.  The specialization resolving Sun's recently stated
Conjecture~4.1 follows from the master identity together with
classical log-sine evaluations.

The unweighted sums considered by Sun are a distinguished member of the
family.  For $0<\alpha<\pi/2$, put
\begin{equation}\label{eq:lambda-def}
 \lambda_\alpha=4\sin^2\alpha,
\end{equation}
and let $D$ denote differentiation with respect to the argument of $f$.
Define
\begin{equation}\label{eq:c-derivative-def}
 c_r=\left.\frac{\dd^r}{\dd a^r}
 \frac{\Gamma(1+a)^2}{\Gamma(1+2a)}\right|_{a=0}
 \qquad(r\geq0).
\end{equation}
Thus $c_0=1$ and $c_1=0$.

For comparison with the harmonic-sum literature, set
\[
 H_n^{(m)}=\sum_{j=1}^n j^{-m},
\]
and let $\mathcal B_r$ denote the $r$th complete exponential Bell
polynomial.  Logarithmic differentiation gives
\begin{equation}\label{eq:harmonic-expansion}
 f^{(r)}(k)=\frac{1}{k\binom{2k}{k}}
 \mathcal B_r\bigl(\eta_1(k),\ldots,\eta_r(k)\bigr),
\end{equation}
where
\begin{align}
 \eta_1(k)&=2\bigl(H_{k-1}-H_{2k-1}\bigr),\label{eq:eta1}\\
 \eta_m(k)&=(-1)^m(m-1)!\bigl((2-2^m)\zeta(m)
 -2H_{k-1}^{(m)}+2^mH_{2k-1}^{(m)}\bigr),\qquad m\ge2.
 \label{eq:etam}
\end{align}
Thus every fixed derivative sum is indeed an inverse-binomial harmonic
sum of a type naturally related to the general frameworks cited above.
Formula~\eqref{eq:harmonic-expansion}, however, becomes increasingly
complicated with $r$; the parameter identity below treats all $r$
simultaneously.

\begin{theorem}[Weighted all-orders formula]\label{thm:weighted}
For every $0<\alpha<\pi/2$ and integer $r\geq0$, the series
\begin{equation}\label{eq:T-def}
 T_r(\alpha)=\sum_{k=1}^{\infty}\lambda_\alpha^{k-1}
 (D+\log\lambda_\alpha)^r f(k)
\end{equation}
converges absolutely and satisfies
\begin{equation}\label{eq:weighted-all-orders}
 T_r(\alpha)
 =\frac{1}{\sin(2\alpha)}
 \left\{
 \alpha c_r-
 \sum_{j=1}^{r}2^{j-1}\binom{r}{j}c_{r-j}
 \Ls_{j+1}(2\alpha)
 \right\}.
\end{equation}
Equivalently, the ordinary weighted derivative sums
\begin{equation}\label{eq:S-def}
 S_r(\alpha)=\sum_{k=1}^{\infty}\lambda_\alpha^{k-1}f^{(r)}(k)
\end{equation}
are given by the finite binomial inversion
\begin{equation}\label{eq:binomial-inversion}
 S_r(\alpha)=\sum_{j=0}^{r}\binom{r}{j}
 (-\log\lambda_\alpha)^{r-j}T_j(\alpha).
\end{equation}
For $r\geq2$, the coefficient $c_r$ is a polynomial in
$\zeta(2),\ldots,\zeta(r)$, and
\begin{equation}\label{eq:c-recurrence-intro}
 c_r=\sum_{m=2}^{r}\binom{r-1}{m-1}
 (-1)^m(2-2^m)(m-1)!\,\zeta(m)c_{r-m}.
\end{equation}
\end{theorem}

The choice $\alpha=\pi/6$ is distinguished by
$\lambda_{\pi/6}=1$.  Thus Theorem~\ref{thm:weighted} immediately gives
the unweighted formula used to resolve Sun's conjecture.

\begin{corollary}[Unweighted all-orders formula]\label{cor:all-orders}
For every integer $r\geq0$,
\begin{equation}\label{eq:all-orders}
 \sum_{k=1}^{\infty}f^{(r)}(k)
 =\frac{\pi}{3\sqrt{3}}c_r
 -\frac{1}{\sqrt{3}}\sum_{j=1}^{r}
   2^j\binom{r}{j}c_{r-j}
   \Ls_{j+1}\left(\frac{\pi}{3}\right).
\end{equation}
\end{corollary}

\begin{corollary}[Sun's Conjecture~4.1]\label{cor:sun}
The following identities hold:
\begin{align}
 \sum_{k=1}^{\infty}f'(k)
   &=-\frac{3}{2}L_{-3}(2), \label{eq:first}\\
 \sum_{k=1}^{\infty}f''(k)
   &=3L_{-3}(3)
     =\frac{4\pi^3}{27\sqrt{3}}, \label{eq:second}\\
 \sum_{k=1}^{\infty}f'''(k)
   &=\frac{3}{4}\left(2\pi^2L_{-3}(2)-27L_{-3}(4)\right).
   \label{eq:third}
\end{align}
\end{corollary}

At the next order, the standard level-$6$ reduction introduces a
depth-two multiple polylogarithmic constant.

\begin{corollary}\label{cor:fourth}
The fourth derivative sum satisfies
\begin{equation}\label{eq:fourth-gl}
 \sum_{k=1}^{\infty}f^{(4)}(k)
 =-72\zeta(3)L_{-3}(2)
  +\frac{134\pi^5}{243\sqrt{3}}
  -32\sqrt{3}\,\Gl_{4,1}\left(\frac{\pi}{3}\right).
\end{equation}
\end{corollary}

The remainder of the paper is organized as follows.  Section~2 proves
the weighted parameter identity, and Section~3 derives the all-orders
formula.  Section~4 establishes the special values required for Sun's
conjecture and the fourth-order evaluation.  Section~5 places the
weighted specializations in the cyclotomic setting and records concrete
level-$4$ and level-$3$ consequences.

\section{The weighted parameter identity}

Set
\[
 q(t)=t(1-t),\qquad 0<t<1.
\]
For $\Re x>0$, the beta integral gives
\begin{equation}\label{eq:beta-f}
 f(x)=\frac12 B(x,x)
     =\frac12\int_0^1 q(t)^{x-1}\dd t.
\end{equation}
For $a\in\mathbb C$, the expressions $q(t)^a$ and
$\lambda_\alpha^a$ use the real logarithms of their positive bases.

\begin{proposition}\label{prop:analytic-weighted}
Fix $0<\alpha<\pi/2$ and write $\lambda=\lambda_\alpha$.  The series
\begin{equation}\label{eq:Phi-def}
 \Phi_\alpha(a)=\sum_{k=1}^{\infty}\lambda^{k+a-1}f(k+a)
\end{equation}
defines a holomorphic function on $\Re a>-1$, and
\begin{equation}\label{eq:Phi-integral}
 \Phi_\alpha(a)=\frac12\int_0^1
 \frac{(\lambda q(t))^a}{1-\lambda q(t)}\dd t.
\end{equation}
Moreover, for every integer $r\geq0$,
\begin{equation}\label{eq:Phi-derivatives}
 \Phi_\alpha^{(r)}(0)
 =\sum_{k=1}^{\infty}\lambda^{k-1}
 (D+\log\lambda)^r f(k),
\end{equation}
and the series on the right converges absolutely.
\end{proposition}

\begin{proof}
Since $0<\lambda<4$, one has $0<\lambda q(t)\leq\lambda/4<1$.
Let $K$ be a compact subset of $\{a:\Re a>-1\}$ and choose
$\delta<1$ such that $\Re a\geq-\delta$ for $a\in K$.  Then
\[
 \sum_{k=1}^{\infty}
 \left|(\lambda q(t))^{k+a-1}\right|
 =\frac{(\lambda q(t))^{\Re a}}{1-\lambda q(t)}
 \leq\frac{\lambda^{-\delta}q(t)^{-\delta}}
 {1-\lambda/4}.
\]
The majorant is integrable on $(0,1)$ because
$q(t)\asymp t$ as $t\to0^+$ and
$q(t)\asymp1-t$ as $t\to1^-$.  Hence
\eqref{eq:beta-f}, Fubini's theorem, and locally uniform dominated
convergence give \eqref{eq:Phi-integral} and holomorphy on
$\Re a>-1$.

Differentiation introduces the factor
$\log^r(\lambda q(t))$.  For every fixed $r$, the preceding majorant
multiplied by $|\log(\lambda q(t))|^r$ remains integrable.  Thus one
may differentiate the integral and series locally uniformly.  In
particular,
\[
 \frac{\dd^r}{\dd a^r}
 \left[\lambda^{k+a-1}f(k+a)\right]_{a=0}
 =\lambda^{k-1}(D+\log\lambda)^r f(k),
\]
which proves \eqref{eq:Phi-derivatives}.  More explicitly,
\[
 \lambda^{k-1}(D+\log\lambda)^r f(k)
 =\frac12\int_0^1(\lambda q(t))^{k-1}
 \log^r(\lambda q(t))\dd t.
\]
Taking absolute values, summing the geometric series, and applying
Tonelli's theorem proves the asserted absolute convergence.
\end{proof}

\begin{proposition}[Weighted master identity]\label{prop:weighted-master}
For $0<\alpha<\pi/2$ and $\Re a>-\tfrac12$,
\begin{equation}\label{eq:weighted-master}
 \Phi_\alpha(a)
 =\frac{1}{\sin(2\alpha)}
  \frac{\Gamma(1+a)^2}{\Gamma(1+2a)}
  \int_0^\alpha(2\sin\theta)^{2a}\dd\theta.
\end{equation}
\end{proposition}

\begin{proof}
Put $s=\sin\alpha$.  Using the symmetry $q(t)=q(1-t)$ and then
substituting $u=4t(1-t)$ on $0\leq t\leq1/2$, equation
\eqref{eq:Phi-integral} becomes
\begin{equation}\label{eq:weighted-hyper-start}
 \Phi_\alpha(a)=\frac14 s^{2a}\int_0^1
 \frac{u^a(1-u)^{-1/2}}{1-s^2u}\dd u.
\end{equation}
The standard formulas used below are recorded in
\cite[(5.5.5), (8.17.7), (15.6.1), and (15.8.1)]{DLMF}.  Euler's
integral representation gives
\begin{equation}\label{eq:weighted-hyper1}
 \Phi_\alpha(a)=\frac14s^{2a}B\left(a+1,\frac12\right)
 {}_2F_1\left(1,a+1;a+\frac32;s^2\right).
\end{equation}
Euler's transformation yields
\begin{equation}\label{eq:weighted-euler}
 {}_2F_1\left(1,a+1;a+\frac32;s^2\right)
 =\frac{1}{\cos\alpha}
 {}_2F_1\left(a+\frac12,\frac12;
                 a+\frac32;s^2\right).
\end{equation}
For $A=a+\tfrac12$, the incomplete-beta identity
\[
 B_z\left(A,\frac12\right)
 =\frac{z^A}{A}
 {}_2F_1\left(A,\frac12;A+1;z\right),
\]
together with $z=\sin^2\theta$, gives
\begin{equation}\label{eq:weighted-incomplete-beta}
 {}_2F_1\left(a+\frac12,\frac12;
                 a+\frac32;s^2\right)
 =2\left(a+\frac12\right)s^{-2a-1}
  \int_0^\alpha\sin^{2a}\theta\dd\theta.
\end{equation}
Substituting \eqref{eq:weighted-euler} and
\eqref{eq:weighted-incomplete-beta} into
\eqref{eq:weighted-hyper1}, and using
$\Gamma(a+\tfrac32)=(a+\tfrac12)\Gamma(a+\tfrac12)$, gives
\[
 \Phi_\alpha(a)
 =\frac{\sqrt\pi\,\Gamma(1+a)}
 {2\sin\alpha\cos\alpha\,\Gamma(a+1/2)}
 \int_0^\alpha\sin^{2a}\theta\dd\theta.
\]
Legendre's duplication formula then reduces this expression to
\eqref{eq:weighted-master}.  All integral representations above are
valid for $\Re a>-\tfrac12$; both sides are holomorphic there.
\end{proof}

\section{Proof of the all-orders formulas}

\begin{proof}[Proof of Theorem~\ref{thm:weighted}]
Write
\[
 Q(a)=\frac{\Gamma(1+a)^2}{\Gamma(1+2a)},
 \qquad
 M_\alpha(a)=\int_0^\alpha(2\sin\theta)^{2a}\dd\theta.
\]
Then the right-hand side of \eqref{eq:weighted-master} is
$Q(a)M_\alpha(a)/\sin(2\alpha)$.
For $|a|<1$, the Taylor expansion of $\log\Gamma(1+a)$,
recorded in \cite[(5.7.3)]{DLMF}, gives
\begin{equation}\label{eq:logQ-series}
 \log Q(a)=\sum_{m=2}^{\infty}
 (-1)^m(2-2^m)\zeta(m)\frac{a^m}{m}.
\end{equation}
Thus $Q^{(r)}(0)=c_r$.  Differentiating
$Q'(a)=(\log Q(a))'Q(a)$ and comparing Taylor coefficients proves
\eqref{eq:c-recurrence-intro}.

For each fixed $j\geq1$, differentiation under the integral sign is
valid near $a=0$.  Indeed, for $|\Re a|\leq1/4$,
\[
 \left|\frac{\partial^j}{\partial a^j}
 (2\sin\theta)^{2a}\right|
 \leq(2\sin\theta)^{-1/2}
 |2\log(2\sin\theta)|^j,
\]
and the majorant is integrable on $(0,\alpha)$.  Hence
\begin{equation}\label{eq:M-alpha-derivative}
 M_\alpha(0)=\alpha,
 \qquad
 M_\alpha^{(j)}(0)
 =-2^{j-1}\Ls_{j+1}(2\alpha)
 \quad(j\geq1).
\end{equation}
Leibniz's rule applied to \eqref{eq:weighted-master}, followed by
Proposition~\ref{prop:analytic-weighted}, proves
\eqref{eq:weighted-all-orders}.  Formula
\eqref{eq:binomial-inversion} is the inverse of the binomial relation
implicit in \eqref{eq:T-def}.  The polynomial assertion follows from
\eqref{eq:logQ-series}, equivalently from
\eqref{eq:c-recurrence-intro}.
\end{proof}

\begin{proof}[Proof of Corollary~\ref{cor:all-orders}]
Set $\alpha=\pi/6$ in Theorem~\ref{thm:weighted}.  Then
$\lambda_\alpha=1$ and $\sin(2\alpha)=\sqrt3/2$, so
\eqref{eq:weighted-all-orders} becomes \eqref{eq:all-orders}.
\end{proof}

The first coefficients are
\begin{equation}\label{eq:first-cs}
 c_0=1,\qquad c_1=0,\qquad c_2=-2\zeta(2),\qquad
 c_3=12\zeta(3),\qquad c_4=-54\zeta(4).
\end{equation}
For the last identity one uses
$\zeta(2)^2=\tfrac52\zeta(4)$ in
\eqref{eq:c-recurrence-intro}.

\section{Special values at cubic arguments}

\begin{lemma}\label{lem:logsine}
The following evaluations hold:
\begin{align}
 \Ls_2\left(\frac{\pi}{3}\right)
   &=\Cl_2\left(\frac{\pi}{3}\right), \label{eq:ls2}\\
 -\Ls_3\left(\frac{\pi}{3}\right)
   &=\frac{7\pi^3}{108}, \label{eq:ls3}\\
 \Ls_4\left(\frac{\pi}{3}\right)
   &=\frac{\pi}{2}\zeta(3)
     +\frac{9}{2}\Cl_4\left(\frac{\pi}{3}\right), \label{eq:ls4}\\
 -\Ls_5\left(\frac{\pi}{3}\right)
   &=\frac{1543\pi^5}{19440}
     -6\Gl_{4,1}\left(\frac{\pi}{3}\right). \label{eq:ls5}
\end{align}
\end{lemma}

\begin{proof}
To justify \eqref{eq:ls2}, integrate the Abel-regularized Fourier
series
\[
 -\log|1-\rho e^{ix}|
 =\sum_{n=1}^{\infty}\frac{\rho^n\cos(nx)}{n},
 \qquad 0<\rho<1.
\]
For fixed $\rho$ the series is uniformly convergent.  As
$\rho\uparrow1$, dominated convergence applies on $(0,\pi/3)$:
for $\rho\geq1/2$ the left-hand side is bounded in absolute value by
$C+|\log x|$, which is integrable at the origin.  Dominated
convergence also applies to the resulting sine series, and therefore
\[
 -\int_0^\phi\log\left(2\sin\frac{x}{2}\right)\dd x
 =\sum_{n=1}^{\infty}\frac{\sin(n\phi)}{n^2}
 =\Cl_2(\phi).
\]
Taking $\phi=\pi/3$ proves \eqref{eq:ls2}.  The evaluations
\eqref{eq:ls3}--\eqref{eq:ls5} are displayed explicitly in
\cite[Example~10]{BorweinStraub2011}; they arise from the symbolic
reduction in \cite[Theorem~3]{BorweinStraub2011}.
\end{proof}

\begin{lemma}\label{lem:clausen-L}
For every integer $m\geq1$,
\begin{align}
 \Cl_{2m}\left(\frac{\pi}{3}\right)
 &=\frac{\sqrt{3}}{2}\left(1+2^{1-2m}\right)L_{-3}(2m),
 \label{eq:clausen-L}\\
 \Cl_{2m}\left(\frac{2\pi}{3}\right)
 &=\frac{\sqrt{3}}{2}L_{-3}(2m).\label{eq:clausen-L-two-thirds}
\end{align}
In particular,
\begin{equation}\label{eq:cl2cl4}
 \Cl_2\left(\frac{\pi}{3}\right)
 =\frac{3\sqrt{3}}{4}L_{-3}(2),
 \qquad
 \Cl_4\left(\frac{\pi}{3}\right)
 =\frac{9\sqrt{3}}{16}L_{-3}(4).
\end{equation}
\end{lemma}

\begin{proof}
For odd $n$,
\[
 \sin\left(\frac{n\pi}{3}\right)
 =\frac{\sqrt{3}}{2}\chi_{-3}(n),
\]
whereas, for $n=2j$,
\[
 \sin\left(\frac{n\pi}{3}\right)
 =\frac{\sqrt{3}}{2}\chi_{-3}(j).
\]
Since $\chi_{-3}(2)=-1$, the even-indexed part of
$L_{-3}(2m)$ is $-2^{-2m}L_{-3}(2m)$; its odd-indexed part is
$(1+2^{-2m})L_{-3}(2m)$.  Splitting the defining series for
$\Cl_{2m}(\pi/3)$ into odd and even indices gives
\[
 \Cl_{2m}\left(\frac{\pi}{3}\right)
 =\frac{\sqrt{3}}{2}
 \left(1+2^{1-2m}\right)L_{-3}(2m).
\]
The second identity follows directly from
$\sin(2\pi n/3)=\tfrac{\sqrt3}{2}\chi_{-3}(n)$.
\end{proof}

\begin{lemma}\label{lem:L3}
One has
\begin{equation}\label{eq:L3}
 L_{-3}(3)=\frac{4\pi^3}{81\sqrt{3}}.
\end{equation}
\end{lemma}

\begin{proof}
The identity
\[
 \chi_{-3}(n)=\frac{2}{\sqrt{3}}
 \sin\left(\frac{2\pi n}{3}\right)
\]
and the Fourier expansion
\[
 \sum_{n=1}^{\infty}\frac{\sin(2\pi nx)}{n^3}
 =\frac{2\pi^3}{3}B_3(x),
 \qquad 0\leq x\leq1,
\]
where $B_3(x)=x^3-\tfrac32x^2+\tfrac12x$, imply
\[
 L_{-3}(3)
 =\frac{2}{\sqrt{3}}\frac{2\pi^3}{3}B_3\left(\frac13\right)
 =\frac{4\pi^3}{81\sqrt{3}}.
\]
\end{proof}

\begin{proof}[Proof of Corollary~\ref{cor:sun}]
For $r=1$, Corollary~\ref{cor:all-orders} and
Lemmas~\ref{lem:logsine} and~\ref{lem:clausen-L} give
\[
 \sum_{k=1}^{\infty}f'(k)
 =-\frac{2}{\sqrt{3}}\Ls_2\left(\frac{\pi}{3}\right)
 =-\frac{3}{2}L_{-3}(2).
\]
For $r=2$, using \eqref{eq:first-cs},
\[
 \begin{aligned}
 \sum_{k=1}^{\infty}f''(k)
 &= -\frac{2\pi}{3\sqrt{3}}\zeta(2)
    -\frac{4}{\sqrt{3}}\Ls_3\left(\frac{\pi}{3}\right)\\
 &= -\frac{\pi^3}{9\sqrt{3}}
    +\frac{7\pi^3}{27\sqrt{3}}
  =\frac{4\pi^3}{27\sqrt{3}}
  =3L_{-3}(3).
 \end{aligned}
\]
Finally, the case $r=3$ gives
\[
 \sum_{k=1}^{\infty}f'''(k)
 =\frac{4\pi}{\sqrt{3}}\zeta(3)
 +\frac{12}{\sqrt{3}}\zeta(2)
   \Ls_2\left(\frac{\pi}{3}\right)
 -\frac{8}{\sqrt{3}}\Ls_4\left(\frac{\pi}{3}\right).
\]
Substitution of Lemmas~\ref{lem:logsine} and
\ref{lem:clausen-L} cancels the two terms containing
$\pi\zeta(3)$ and yields
\[
 \sum_{k=1}^{\infty}f'''(k)
 =\frac{3}{2}\pi^2L_{-3}(2)-\frac{81}{4}L_{-3}(4),
\]
which is \eqref{eq:third}.
\end{proof}

\begin{proof}[Proof of Corollary~\ref{cor:fourth}]
Corollary~\ref{cor:all-orders}, \eqref{eq:first-cs}, and
Lemma~\ref{lem:logsine} give
\[
 \sum_{k=1}^{\infty}f^{(4)}(k)
 =-72\zeta(3)L_{-3}(2)
  -\frac{97\pi^5}{135\sqrt{3}}
  -\frac{16}{\sqrt{3}}\Ls_5\left(\frac{\pi}{3}\right).
\]
Substituting \eqref{eq:ls5} and using
\[
 -\frac{97}{135}+\frac{16\cdot1543}{19440}
 =\frac{134}{243}
\]
proves \eqref{eq:fourth-gl}.
\end{proof}

\begin{remark}\label{rem:structural}
Borwein and Straub describe $\Gl_{4,1}(\pi/3)$ as the first
``presumed-irreducible'' value arising in their reduction
\cite[Example~10]{BorweinStraub2011}.  Accordingly, the standard
reduction of \eqref{eq:all-orders} changes character at $r=4$: the
cases $r\leq3$ reduce to ordinary zeta and Dirichlet $L$-values,
whereas the fourth-order case introduces a depth-two constant.  No
claim of algebraic independence or irreducibility is required here.
\end{remark}

\section{Cyclotomic specializations and low-level examples}

The case $r=0$ of Theorem~\ref{thm:weighted} is classical.  Writing
$x=\sin\alpha$ gives
\begin{equation}\label{eq:classical-generating}
 \sum_{k=1}^{\infty}\frac{(4x^2)^k}{k\binom{2k}{k}}
 =\frac{2x\arcsin x}{\sqrt{1-x^2}},\qquad |x|<1,
\end{equation}
which belongs to the classical theory of reciprocal central-binomial
series; see, for example, Lehmer \cite{Lehmer1985}.  The new information
in Theorem~\ref{thm:weighted} is the gamma deformation and its
simultaneous differentiation to arbitrary order.

For an integer $N\geq3$, set
\begin{equation}\label{eq:N-defs}
 \alpha_N=\frac{\pi}{N},\qquad
 \lambda_N=4\sin^2\frac{\pi}{N},\qquad
 \zeta_N=e^{2\pi i/N}.
\end{equation}

\begin{corollary}\label{cor:cyclotomic}
For every $N\geq3$ and $r\geq0$,
\begin{equation}\label{eq:N-formula}
 \begin{aligned}
 &\sum_{k=1}^{\infty}\lambda_N^{k-1}
 (D+\log\lambda_N)^r f(k)\\
 &\quad=\frac{1}{\sin(2\pi/N)}
 \left\{
 \frac{\pi}{N}c_r-
 \sum_{j=1}^{r}2^{j-1}\binom{r}{j}c_{r-j}
 \Ls_{j+1}\left(\frac{2\pi}{N}\right)
 \right\}.
 \end{aligned}
\end{equation}
\end{corollary}

\begin{proof}
This is Theorem~\ref{thm:weighted} with $\alpha=\pi/N$.
\end{proof}

\begin{remark}[Relation with cyclotomic multiple zeta values]
Let $\mathcal C_N$ denote the $\mathbb Q$-algebra generated by $\pi$,
ordinary multiple zeta values, and the real and imaginary parts of
multiple polylogarithms at $\zeta_N$.  The recursive reduction of
Borwein and Straub \cite[Theorem~3]{BorweinStraub2011} implies
\[
 \Ls_m\left(\frac{2\pi}{N}\right)\in\mathcal C_N
 \qquad(m\ge2).
\]
Consequently, the shifted sum in \eqref{eq:N-formula} belongs to
$\mathcal C_N$; the unshifted sums obtained from
\eqref{eq:binomial-inversion} additionally involve powers of
$\log\lambda_N=\log|1-\zeta_N|^2$.  This cyclotomic membership is
consistent with, and should be viewed in the context of, the broader
methods of Au, Zhou, Sun--Zhou, and Campbell--Glasser--Zhou
\cite{Au2020,Au2022,Zhou2023,SunZhou2024,CampbellGlasserZhou2024}.  No novelty
is claimed here for the general cyclotomic reduction itself; the point
of \eqref{eq:N-formula} is its explicit continuous origin and its
validity for all derivative orders at once.
\end{remark}

Let
\[
 G=\sum_{n=0}^{\infty}\frac{(-1)^n}{(2n+1)^2}
\]
denote Catalan's constant.

\begin{corollary}[The specialization $\alpha=\pi/4$]\label{cor:catalan}
One has
\begin{align}
 T_0\left(\frac\pi4\right)&=\frac\pi4,
 \label{eq:catalan-zero}\\
 T_1\left(\frac\pi4\right)&=-G,
 \label{eq:catalan-shifted-first}\\
 T_2\left(\frac\pi4\right)&=-\frac{\pi^3}{12}
 -2\Ls_3\left(\frac\pi2\right),
 \label{eq:catalan-second}\\
 T_3\left(\frac\pi4\right)&=3\pi\zeta(3)+\pi^2G
 -4\Ls_4\left(\frac\pi2\right).
 \label{eq:catalan-third}
\end{align}
In particular,
\begin{equation}\label{eq:catalan-first}
 \sum_{k=1}^{\infty}2^{k-1}f'(k)
 =-G-\frac\pi4\log2.
\end{equation}
\end{corollary}

\begin{proof}
Take $\alpha=\pi/4$, so that $\lambda_\alpha=2$ and
$\sin(2\alpha)=1$.  The formulas follow from
\eqref{eq:weighted-all-orders}, \eqref{eq:first-cs}, and
$\Ls_2(\pi/2)=\Cl_2(\pi/2)=G$.  Finally,
$S_1=T_1-(\log2)T_0$ gives \eqref{eq:catalan-first}.
\end{proof}

The remaining log-sine values in
\eqref{eq:catalan-second}--\eqref{eq:catalan-third} are standard
level-$4$ cyclotomic constants; they may be reduced algorithmically to
multiple polylogarithms at fourth roots of unity by the general methods
of Borwein--Straub and Au \cite{BorweinStraub2011,Au2020,Au2022}.
We leave them in log-sine form, which is the natural form supplied by
Theorem~\ref{thm:weighted}.

\begin{corollary}[The specialization $\alpha=\pi/3$]\label{cor:level3}
One has
\begin{align}
 \sum_{k=1}^{\infty}3^{k-1}f(k)
 &=\frac{2\pi}{3\sqrt3},\label{eq:level3-zero}\\
 \sum_{k=1}^{\infty}3^{k-1}(D+\log3)f(k)
 &=-L_{-3}(2),\label{eq:level3-shifted-first}\\
 \sum_{k=1}^{\infty}3^{k-1}f'(k)
 &=-L_{-3}(2)-\frac{2\pi}{3\sqrt3}\log3.
 \label{eq:level3-first}
\end{align}
\end{corollary}

\begin{proof}
Set $\alpha=\pi/3$ in Theorem~\ref{thm:weighted}.  Since
$\lambda_\alpha=3$ and $\sin(2\alpha)=\sqrt3/2$, the case $r=0$
gives \eqref{eq:level3-zero}, while the case $r=1$, together with
\eqref{eq:ls2} and \eqref{eq:clausen-L-two-thirds}, gives
\eqref{eq:level3-shifted-first}.  Binomial inversion yields
\eqref{eq:level3-first}.
\end{proof}

\section*{Statements and declarations}

\noindent\textbf{Funding.}
No specific funding was received for this work.

\medskip
\noindent\textbf{Competing interests.}
The author declares no competing interests relevant to the content of
this article.

\medskip
\noindent\textbf{Data availability.}
No datasets were generated or analyzed in this study.

\end{document}